\documentclass{article}
\usepackage{amsmath,amssymb,amsthm}

\newcommand{\R}{\mathbb{R}}

\newcommand{\N}{\mathbb{N}}
\newcommand{\E}{\mathbf{E}}

\newcommand{\Gc}{\mathcal{G}}
\newcommand{\Hc}{\mathcal{H}}

\newcommand{\var}{\textnormal{Var}}
\newcommand{\bin}{\textnormal{Bin}}
\newcommand{\ber}{\textnormal{Ber}}

\newcommand{\plog}{\text{polylog}}
\newcommand{\poly}{\text{poly}}
\newcommand{\Q}{\mathbf{Q}}

\newcommand{\pr}{\textnormal{Pr}}

\newcommand{\m}{\textnormal{m}}
\newcommand{\Pn}{\textnormal{P}}
\newcommand{\Cn}{\textnormal{C}}

\renewcommand{\part}{\text{part}}

\renewcommand{\epsilon}{\varepsilon}
\renewcommand{\theta}{\vartheta}
\renewcommand{\phi}{\varphi}
\renewcommand{\P}{\mathbf{P}}
\renewcommand{\Q}{\mathbf{Q}}

\renewcommand{\d}{\textnormal{d}}

\newcommand{\lb}{\left[}
\newcommand{\rb}{\right]}

\newtheorem{thm}{Theorem}[section]
\newtheorem{lem}[thm]{Lemma}
\newtheorem{cor}[thm]{Corollary}
\newtheorem{conj}[thm]{Conjecture}

\newtheorem{prop}[thm]{Proposition}

\title{Optimal couplings between sparse block models}
\author{James Hirst \\ Department of Mathematics, MIT}

\begin{document}
\maketitle

\begin{abstract}
We study the problem of coupling a stochastic block model with a planted bisection to a uniform random graph having the same average degree. Focusing on the regime where the average degree is a constant relative to the number of vertices $n$, we show that the distance to which the models can be coupled undergoes a phase transition from $O(\sqrt{n})$ to $\Omega(n)$ as the planted bisection in the block model varies. This settles half of a conjecture of Bollob\'{a}s and Riordan and has some implications for sparse graph limit theory. In particular, for certain ranges of parameters, a block model and the corresponding uniform model produce samples which must converge to the same limit point. This implies that any notion of convergence for sequences of graphs with $\Theta(n)$ edges which allows for samples from a limit object to converge back to the limit itself must identify these models.
\end{abstract}

\section{Introduction}
The idea of putting a metric on the space of graphs has generated a significant amount of fruitful research in the last decade. In particular, this has led to an extremely rich theory of dense graph limits, a line of work stemming from \cite{BCLSV08} which remains active today. Here, the word \emph{dense} refers to the scaling in the number of edges: a dense sequence of graphs with $n$ vertices should have $\Theta(n^2)$ edges, and any set of $o(n^2)$ edges doesn't contribute to the limit of the sequence (they become measure 0 sets in the limit). The metric that best captures similarity between dense graphs is the \emph{cut distance}, which, informally, is a measure of how well one can align two graphs to have a small number of edges across every cut. Closeness in cut distance turns out to be equivalent to a number of natural metrics including subgraph counts and free energies. For an excellent exposition of these results, we refer the reader to \cite{Lov12}.

Modulo some technical conditions, a suitably re-normalized version of the cut distance is also a valid and interesting metric for graphs with with $o(n^2)$ edges. In \cite{BR09}, the authors show that for graph sequences in which every $o(n)$ vertices of a graph span $o(n)$ edges (forbidding graphs like a clique on $o(n)$ vertices), many of the nice properties of the cut distance continue to hold. The more recent work of \cite{BCCZ19} uses a relaxation of this condition based on boundedness in an $\ell_p$ sense (ordinary graphs can be thought of as bounded in $\ell_\infty$). However, in both of these results is an implicit requirement that graphs have unbounded degrees as $n$ grows, i.e., $\omega(n)$ edges.

An attempt at treating the $\Theta(n)$ sparsity regime is given in \cite{BR11}. One of the main difficulties here is that the random graph $G(n,c/n)$ concentrates much more weakly than usual. Indeed, the degree of a vertex has mean and variance of the same order $c$, and many key quantities have tails which decay only exponentially in $n$, rather than super-exponentially. As such, basic results which follow from a union bound for denser graphs break down. In this setting, the cut distance turns out to be much too strong of a metric: even a sequence of i.i.d. graphs from $G(n,c/n)$ fails to be a Cauchy sequence. On the other hand, local metrics based on subgraph counts are in some sense too weak, as they fail to distinguish between a random graph and a random bipartite graph of the same average degree.

The most basic object in dense graph limit theory is the graphon, a symmetric and measurable function $W:[0,1]^2 \to [0,1]$. Given a graphon $W$, there is a natural way to sample an $n$-vertex graph $G_n$: choose $x_1,\ldots,x_n$ uniformly from $[0,1]$ and independently place each edge $(i,j)$ with probability $W(x_i,x_j)$. This produces a graph with $\Theta(n^2)$ edges (assuming $W \neq 0$), and such models are well studied. For instance a basic result from graph limit theory says that the sequence $\{G_n\}$ converges to $W$ (under the cut distance) almost surely. 

With a slight modification, we can use this construction to sample graphs with $\Theta(n)$ edges. To avoid confusion, we will call a symmetric and measurable function $\kappa:[0,1]^2 \to \R$ a \emph{kernel} to distinguish it from a graphon which typically maps to $[0,1]$. Given a bounded kernel $\kappa:[0,1]^2 \to \R$, again take $x_1,\ldots,x_n$ as i.i.d. $U(0,1)$ but now place edge $(i,j)$ with probability $\kappa(x_i,x_j)/n$ (projecting this value into $[0,1]$ as required). Denote this random graph model by $G(n,\kappa)$.

Can we still say in some sense that the (sparsely) sampled sequence $\{G_n\}$ converges to $\kappa$? We will see momentarily that there is a natural barrier to this which is not present in the dense case. Given $n$-vertex graphs $G$ and $H$ on the same vertex set, we define the \emph{edit distance} between $G$ and $H$ to be the minimum number of edges which need to be changed to turn $H$ into an isomorphic copy of $G$, and we write
\[ \d_1(G,H) = \min_{G' \cong G} \left|E(G') \triangle E(H) \right|. \]
As in the dense case, it is reasonable to assume that any notion of convergence for sparse graphs should be invariant up to edit distance $o(n)$. Following $\cite{BR11}$, we will call the models $G(n,\kappa_1)$ and $G(n,\kappa_2)$ (or sometimes just $\kappa_1$ and $\kappa_2$) \emph{essentially equivalent} if there is a coupling $\mu$ between $G(n,\kappa_1)$ and $G(n,\kappa_2)$ such that $\E_\mu \d_1(G,H) = o(n)$, where $(G,H) \sim \mu$. We will call them \emph{essentially different} otherwise. Returning to our question above, if we can find $\kappa_1 \neq \kappa_2$ which are essentially equivalent, then samples from $G(n,\kappa_1)$ and $G(n,\kappa_2)$ must converge in some sense to both $\kappa_1$ and $\kappa_2$.

Consider the following example. Let $\kappa_1$ be identically $c$, and take $\kappa_2$ to have value $c+\delta$ on $[0,1/2]\times[1/2,1]$ and $c-\delta$ elsewhere. Henceforth we will set $\Q = G(n,\kappa_1)$ and $\P = G(n,\kappa_2)$, suppressing the dependence on $n$, $c$, and $\delta$ wherever it avoids confusion. In \cite{BR11}, the authors give a simple example to show that when $c \geq \delta$ and $0 \leq \delta \leq 1$, $\P$ and $\Q$ are in fact essentially equivalent. The following extension of this is half of Conjecture 6.3 from \cite{BR11}.

\begin{conj}\label{conj:br}
Let $0 \leq \delta \leq c$. If $\delta < \sqrt{c}$, then $\P$ and $\Q$ are essentially equivalent.
\end{conj}

Note that here the underlying kernels are truly different: One cannot compose $\kappa_1$ and $\kappa_2$ with (potentially different) measure preserving maps to make them equal almost everywhere. This makes the idea that $\P$ and $\Q$ can none the less be essentially equivalent somewhat surprising at first glace.

Our main result is a proof of Conjecture \ref{conj:br}. In fact, we prove an even stronger result: when $\delta < \sqrt{c}$, the two models can be coupled to an expected edit distance of $O(\sqrt{n})$.

\begin{thm}\label{thm:ub}
Let $0 \leq \delta \leq c$ and denote by $\Pi(\P,\Q)$ the space of coupling measures between $\P$ and $\Q$. If $\delta < \sqrt{c}$, then there exists a coupling $\mu \in \Pi(\P_n,\Q_n)$ such that
\[ \E_{\mu}\left[\d_1(G,H)\right] = O(\sqrt{n}). \]
\end{thm}

We also give a lower bound on the expected edit distance under any coupling.

\begin{thm}\label{thm:lb}
Let $0 \leq \delta \leq c$. For any coupling $\mu \in \Pi(\P,\Q)$,
\[ \E_{\mu}\left[\d_1(G,H)\right] = \Omega\left(\frac{c^{\log(n)^{1/3}}}{\log(n)^{1/3}}\right).\]
\end{thm}

\section{Preliminaries}
\subsection{Optimal transport}
The proof strategy for both Theorem \ref{thm:ub} and Theorem \ref{thm:lb} is to obtain bounds on the quantity
\begin{equation}
\inf_\mu \E_{\mu}\left[\d_1(G,H)\right],
\end{equation}
where the infimum is taken over $\mu \in \Pi(\P,\Q)$. This is actually a well known type of problem called an optimal transport problem (for the cost $d_1$) and as such we will make use of a number of tools from the area. For an excellent survey of optimal transport, see \cite{San15}. The main item we will make use of is the following duality formula.

\begin{prop}[Monge-Kantorovich Duality]\label{prop:dual}
We have the equality
\begin{equation}\label{eq:dual}
\inf_\mu \E_\mu \lb \d_1(G,H) \rb = \sup_f \left| \E_P f - \E_Q f \right|,
\end{equation}
where $\mu \in \Pi(\P,\Q)$, $(G,H) \sim \mu$, and $f$ is 1-Lipschitz in the edit distance.
\end{prop}

This is in fact a special case of Monge-Kantorovich duality for cost functions which satisfy the triangle inequality. Note that Proposition \ref{prop:dual} immediately implies the useful inequality
\[ \inf_\mu \E_\mu \lb \d_1(G,H) \rb \geq \sup_f \left| \E_P f - \E_Q f \right|, \]
which is often referred to as \emph{weak duality}.

It is not hard to see that the infimum in \eqref{eq:dual} is always attained: the space $\Pi(\P,\Q)$ is compact for the weak convergence in duality with the space of continuous functions. Hence, by the definition of weak convergence, the functional sending $\mu$ to $\E_\mu \lb d_1(G,H)\rb$ is continuous. This is enough to guarantee the existence of a minimizer $\mu^*$ (Weierstrass). As it turns out, the supremum is also attained, but the proof is somewhat more involved (see Section 1.6 in \cite{San15}).

\begin{prop}\label{prop:inf}
There exists a coupling $\mu^* \in \Pi(\P,\Q)$ and a 1-Lipschitz function $f^*$ such that
\[ \E_{\mu^*} \lb \d_1(G^*,H^*) \rb = \left| \E_P f^* - \E_Q f^* \right|, \]
where $(G^*,H^*) \sim \mu^*$.
\end{prop}

\subsection{Stochastic block models}
The study of sparse random graphs has also seen some attention from a very different group, motivated by the problem of detecting hidden communities planted in a random model. The stochastic block model for two balanced communities is usually defined as follows. Take parameters $a > b > 0$ (for a ferromagnetic block model) and assign to each vertex $i$ of an $n$ vertex graph a random spin $\sigma_i \in \{-1,1\}$. Then, connect edges $i$ and $j$ with probability $a/n$ if $\sigma_i \sigma_j = 1$, and with probability $b/n$ otherwise, independently for every edge. Notice that samples from the stochastic block model are identical to samples from $\P$, with $a = c+\delta$ and $b = c-\delta$.

Much of the previous work on such models is algorithmic, but some recent theoretical results have been obtained which are extremely relevant to the problem at hand. Of particular interest is \cite{MNS15}, in which the authors settle a conjecture from \cite{DKMZ11} and give an information-theoretic lower bound for the detection problem. This lower bound comes in the form of a contiguity result.

\begin{thm}[Theorem 2.4 from \cite{MNS15}]\label{prop:cont}
When $\delta < \sqrt{c}$, the models $\P_n$ and $\Q_n$ are mutually contiguous, i.e., for a sequence of events $A_n$ we have
\[ \P_n(A_n) \to 0 \iff \Q_n(A_n) \to 0. \]
\end{thm}

To be precise, contiguity is a notion that should be applied to sequences of probabilities, and so when we say the models $\P_n$ and $\Q_n$ are mutually contiguous, we really mean that the sequences $\{\P_n\}$ and $\{\Q_n\}$ are.

An easy but important corollary which shall be our main point of contact with Theorem \ref{prop:cont} is that it is impossible to consistently distinguish between samples from $\P$ and $\Q$ below the contiguity threshold.

\begin{cor}\label{cor:est}
Generate a random pair $(\sigma,G)$, where $\sigma \in \{0,1\}$ and $G$ is an $n$-vertex graph as follows. Take $\sigma \sim \ber(1/2)$ and sample $G$ from $\P$ if $\sigma = 1$, sampling $G$ from $\Q$ otherwise. When $\delta < \sqrt{c}$, there is no estimator $\pi$ taking an $n$-vertex graph to $\{0,1\}$ such that
\begin{equation}\label{eq:est}
\Pr\lb \pi(G) = \sigma \rb = 1 - o(1).
\end{equation}
\end{cor}

\begin{proof}
Given such a $\pi$, take $A_n = \{G: \pi(G) = 1\}$. Then $\P_n(A_n) \to 1$ while $\Q_n(A_n) \to 0$.
\end{proof}

We often call an estimator $\pi$ satisfying \eqref{eq:est} an estimator for the block model or simply an estimator for $\P$. It is known that such estimators exist for all $\delta > \sqrt{c}$ and can even be efficiently computed (see \cite{MNS15} and \cite{MNS18}).

\section{Proofs}
At a high level, the proof strategy for Theorem \ref{thm:ub} is to use Monge-Kantorovich duality along with an upper bound for
\[ \sup_f \left| \E_P f - \E_Q f \right|. \]
If $\left| \E_P f - \E_Q f \right|$ is too large for some $f$, then $f$ can be used as an estimator for $\P$ by thresholding. In order to make this work, we need to know that $f$ is sufficiently concentrated around its mean. 

One can try to apply Azuma-Hoeffding using the vertex exposure martingale, but the martingale differences can only be bounded by the maximum degree of the graph which is a bit less than $\log(n)$ under both $\P$ and $\Q$. This gives a tail bound of the form
\[ \Pr \lb|f - \E f| > t \rb < 2\exp\left(-\frac{2t^2}{n\log(n)^2}\right) \]
at best, losing a $\log(n)^2$ factor in the variance. Rather than looking for tails with the correct variance, we will instead use the following inequality which bounds the variance directly.

\begin{prop}[Efron-Stein \cite{Ste86}]\label{prop:es}
Let $X_1,\ldots,X_n$ be independent random variables and set $X = (X_1,\ldots,X_n)$. Denote by $X^{(i)}$ the random variable obtained from $X$ by re-sampling $X_i$. Then we have
\[ \var(f(X)) \leq \frac{1}{2}\sum_{i=1}^n\E\left[\left(f(X) - f(X^{(i)})\right)^2\right]. \]
\end{prop}

With Efron-Stein, we can prove concentration for Lipschitz functions for samples from any bounded kernel.

\begin{lem}\label{lem:conc}
Let $\kappa$ be a bounded kernel with $d = \sup \kappa$, take $G_n \sim G(n,\kappa)$, and let $f: \Gc_n \to \R$ be 1-Lipschitz in the edit distance. Then we have
\[ \var(f(G_n)) = O(n). \]
\end{lem}

\begin{proof}
The proof is a direct application of Efron-Stein. With $m = \binom{n}{2}$, $f(G_n)$ is a function of $n + m$ independent $U(0,1)$ random variables, $n$ type (or vertex) variables and $m$ edge variables. Notice that for a type variable $i$, $f(X) - f(X^{(i)})$ is stochastically dominated by a $\bin(n,d/n)$ random variable since $f$ is Lipschitz. So we can bound the marginal variance as
\[ \E\left[\left(f(X) - f(X^{(i)})\right)^2\right] \leq \var(\bin(n,d/n)) = d\left(1-\frac{d}{n}\right) \leq d. \]
Similarly for an edge variable $j$, $f(X) - f(X^{(j)})$ is stochastically dominated by a $\ber(d/n)$ random variable which gives
\[ \E\left[\left(f(X) - f(X^{(j)})\right)^2\right] \leq \var(\ber(d/n)) = \frac{d}{n}\left(1-\frac{d}{n}\right) \leq \frac{d}{n}. \]
Putting these together yields
\[ \var(f(G_n)) \leq \frac{1}{2}\left(dn + \frac{d}{n}m\right) \leq dn. \]
\end{proof}

We're now ready to prove Theorem \ref{thm:ub}.

\begin{proof}[Proof of Theorem \ref{thm:ub}]
Suppose that in \eqref{eq:dual} we have
\[ \sup_f\left| \E_\P f - \E_\Q f \right| = \omega(\sqrt{n}). \]
By Proposition \ref{prop:inf} there exists a 1-Lipschitz function $f$ such that
\[ \left| \E_\P f - \E_\Q f \right| = \omega(\sqrt{n}) = \alpha(n)\sqrt{n}, \]
where $\alpha(n)$ is a function with $\lim_{n\to\infty}\alpha(n) = \infty$. Using Lemma \ref{lem:conc} and Chebyshev's inequality we have
\begin{equation}\label{eq:conc}
\pr \left[ |f - \E f| > \sqrt{\alpha(n)}\sqrt{(c+
\delta)n}\right] \leq \frac{1}{\alpha(n)} \to 0
\end{equation}
for both $\P$ and $\Q$. Now, one of the terms $\E_\P f$ and $\E_\Q f$ must be at least the difference $\alpha(n)\sqrt{n}$. For every $n$, define the estimator $\pi$ as follows. If $E_{\P} f > \E_{\Q} f$ then let
\[ \pi(G) = \left\{ \begin{array}{cc}
   1  & \text{ if} \ f(G) \geq \E_{\P} f - \frac{\alpha(n)\sqrt{n}}{2} \\
   0  & \text{otherwise}
\end{array}\right.. \]
Similarly if $\E_{\Q} f > \E_{\P} f$ let
\[ \pi(G) = \left\{ \begin{array}{cc}
   0  & \text{ if} \ f(G) \geq \E_{\Q} f - \frac{\alpha(n)\sqrt{n}}{2} \\
   1  & \text{otherwise}
\end{array}\right.. \]
It follows from \eqref{eq:conc} that $\pi$ is an estimator for $\P$ which is impossible by Corollary \ref{cor:est}.

Returning to \eqref{eq:dual} we get the upper bound
\[ \inf_\mu \E_\mu \lb \d_1(G,H) \rb = O(\sqrt{n}). \]
Again referring to Proposition \ref{prop:inf}, there exists a coupling $\mu\in \Pi(\P,\Q)$ such that
\[ \E_\mu \lb \d_1(G,H) \rb = O(\sqrt{n}). \]
\end{proof}

We should note that it is important to make the distinction between $\E_\P f > \E_\Q f$ and $\E_\Q f > \E_\P f$ for every $n$, as even though $f$ concentrates around its mean for a fixed $n$, the values of $\E_\P f$ and $\E_\Q f$ might not be converging as $n$ grows.

Let us now turn to the proof of Theorem \ref{thm:lb}.

\begin{proof}[Proof of Theorem \ref{thm:lb}]
From weak duality, for any 1-Lipschitz function $f$ we have the inequality
\[ \inf_\mu \E_\mu \lb \d_1(G,H) \rb \geq \left| \E_\P f - \E_\Q f \right|. \]
Given a graph $G$ and a parameter $k\in \N$, define $Y_k$ to be the maximum number of $k$-cycles in a disjoint cycle packing of $G$. Set $f(G) = Y_k$, noting that $f$ is certainly 1-Lipschitz, as an edge can contribute to a single cycle in any disjoint packing.

To remains to control $\E_\P f$ and $\E_\Q f$. To this end, define $X_k$ to be simply the number of (potentially overlapping) $k$-cycles in $G$, counted as ordered sets of $k$ vertices, up to automorphism. Here, we have
\[ \E_\Q X_k = \frac{n(n-1)\cdots(n-k+1)}{2k}\left(\frac{c}{n}\right)^k = \frac{c^k}{2k} + o(1), \]
where the second equality holds for $k = o(\sqrt{n})$. Indeed, there are $\sim n^k/2k$ ways to choose the vertices along with an ordering for the cycle modulo automorphism, and a probability $(c/n)^k$ of the edges appearing in sequence. Extending this result to $\P$, it is shown in \cite{MNS15} (Section 3) that
\[ \E_\P X_k = \frac{c^k + \delta^k}{2k} + o(1), \]
again for $k = o(\sqrt{n})$.

Consider the set $\Hc$ of all graphs $H$ which can be obtained as the union of two distinct, non-disjoint $k$-cycles. Note that every graph $H \in \Hc$ has $m \leq 2k$ vertices and at least $m+1$ edges, since it cannot be a cycle itself. For every such $H$, define $X_H$ to be the number of injective homomorphisms from $H$ into $G$ . We can easily compute
\[ \E_\Q X_H \leq n^m \left(\frac{c}{n}\right)^{m+1} \leq \frac{c^{2k+1}}{n}. \]

Define $Z_k$ by
\[ Z_k = \sum_{H \in \Hc} X_H. \]
Relating $X_k$ and $Y_k$ through $Z_k$, we have the inequalities
\[ X_k \geq Y_k \geq X_k - (2k)^k Z_k, \]
since every embedding of a graph $H \in \Hc$ into $G$ contributes at most $(2k)^k$ $k$-cycles to $X_k$, and removing all such edges yields a graph with $X_k = Y_k$. A graph with $2k$ vertices has at most $(2k)^2/2 = 2k^2$ edges so, counting over all possible labelled graphs, we get $|\Hc| \leq 2^{2k^2}$, where graphs on $<2k$ vertices are identified with graphs on $2k$ vertices with isolated vertices. This gives
\[ \E_\Q Z_k \leq \frac{2^{2k^2}c^{2k+1}}{n}. \]
Taking $k = \log(n)^{1/3}$ (in fact, any $k = \log(n)^{\alpha}$, $\alpha < 1/2$ will work), the numerator is $o(n)$ and so $\E_\Q Z_k = o(1)$. Since the number of copies of any graph in $\P$ is stochastically dominated by the number of copies in a $G(n,(c+\delta)/n)$, $\E_\P Z_k = o(1)$ as well. Putting everything together, we end up with
\[ \left| \E_\P f - \E_\Q f \right| \geq \frac{\delta^{k}}{2k} - o(1) = \Omega\left(\frac{c^{\log(n)^{1/3}}}{\log(n)^{1/3}}\right). \]
\end{proof}

As noted, taking $k = \log(n)^{\alpha}$ for any $\alpha < 1/2$ improves the bound in Theorem \ref{thm:lb} to
\[ \frac{c^{\log(n)^{\alpha}}}{\log(n)^{\alpha}}, \]
but the value of this function remains somewhere between $\plog(n)$ and $\poly(n)$ for all $\alpha < 1/2$ and so this is not a particularly meaningful improvement.

\section{Concluding remarks}\label{sec:rem}
As mentioned, Conjecture \ref{conj:br} is only half of the conjecture made in \cite{BR11}. In addition, the authors conjecture that for $\delta > \sqrt{c}$ and $c > 1$ (to ensure a giant component), the models $\P$ and $\Q$ are essentially different. This means that it should be impossible to couple $\P$ and $\Q$ to an expected edit distance of $o(n)$. Now if we allow $\delta > A\sqrt{c}$ for some sufficiently large constant $A$, then this is certainly the case. This is because the minimum bisection in a graph $G \sim \Q$ is of the order $cn/4 - \Theta(\sqrt{c}n)$ with high probability, whereas $G \sim \P$ almost always has a bisection with $cn/4 - \Theta(\delta n/4)$ edges. Plugging the minimum bisection into \eqref{eq:dual} gives the desired $\Theta(n)$ lower bound.

It is not clear at all that we should be able to take $A = 1$ in the above. One might hope that for $\delta >\sqrt{c}$, the minimum bisections of $\P$ and $\Q$ would differ by $\Theta(n)$. One of the main factors mitigating progress in this area is the complexity of even computing the minimum bisection of a graph $G \sim \Q$. Writing $\m(G)$ for the value of the minimum bisection of $G$, the authors of \cite{DMS17} recently determined the second order growth of $m(G)$ through the formula
\begin{equation}\label{eq:minb}
\E_\Q\lb \m(G)\rb = n\left(\frac{c}{4} + \frac{\Pn^*}{2}\sqrt{c} + o(\sqrt{c})\right) + o(n).
\end{equation}
The constant $\Pn^* \approx 0.7632$ here is the ground state energy of the Sherrington-Kirkpatrick spin glass model, and has no known simple closed form. A similar formula for $\E_\P\lb \m(G) \rb$ is given in \cite{Sen18} with an even more mysterious constant $\Cn^*$. It would be a significant breakthrough to pinpoint the value of $\delta$ for which the minimum bisections in $\P$ and $\Q$ transition from having a difference $O(\sqrt{n})$ to $\Theta(n)$. It could even be that two separate phase transitions occur.

The work of \cite{JMRT16} uses some strong heuristic methods from statistical physics to make predictions in this direction. Although non-rigorous, these heuristics (including the so-called cavity method) have predicted a number of results which have since been proven rigorously (like Theorem \ref{prop:cont} above). They study an SDP relaxation of minimum bisection and predict that it detects communities down to $\delta = \sqrt{c}(1 + (8c)^{-1} + O(c^{-2}))$. This lower bound is at least partially vindicated by a matching upper bound from $\cite{MS16}$, who show that the SDP based estimator in fact does detect communities at $\delta = \sqrt{c} + \epsilon$ for any $\epsilon$, provided $c = c(\epsilon)$ can be taken sufficiently large. It seems as though minimum bisection should not do any worse than the SDP relaxation, which would allow us to take $A = 1 + o_c(1)$ above, but this is not a proof by any means.

Ignoring the question of bisections, it may well be that the models $\P$ and $\Q$ do in fact become essentially different at the contiguity threshold $\delta = \sqrt{c}$. Our choice of $f$ in the proof of Theorem \ref{thm:lb} was based on the simple fact that $\P$ is biased towards having more $k$-cycles than $\Q$ on average. Taking $k$ larger gives a better lower bound, but one cannot take $k$ past $o(\sqrt{\log(n)})$ without overlapping cycles contributing to the expectation. To normalize $f$ into a Lipschitz function would then require an analysis of this overlap and seems difficult. A radically different choice of $f$ may be able to give improved lower bounds. For $\delta < \sqrt{c}$, it seems as though $\Theta(\sqrt{n})$ should be the correct answer: If not, then the minimum bisections in $\P$ and $\Q$ are actually mixed to their second order terms, which would be somewhat surprising.

As a final remark, we should note that the notion of being essentially different is robust under $o(n)$ adversarial edits to our samples. If $\P$ and $\Q$ become essentially different when $\delta > \sqrt{c}$, then Monge-Kantorovich duality along with Lemma \ref{lem:conc} implies the existence of a \emph{robust} estimator for $\P$. As far as we are aware, no estimator which achieves the $\delta > \sqrt{c}$ threshold is known to be robust to $o(n)$ edits. For instance, one can remove all cycles of length $k = o(\log(n))$ with $o(n)$ edits, breaking any current cycle-based estimators which only count short cycles.

\newpage

\bibliographystyle{alpha}
\bibliography{couplings}

\end{document}